\documentclass[12pt,a4paper,psamsfonts]{amsart}

\usepackage{fancyhdr}
\usepackage{appendix}
\usepackage{amssymb,amscd,amsxtra,calc}

\usepackage{mathrsfs}
\usepackage{cmmib57}
\usepackage{multirow}
\usepackage[all]{xy}
\usepackage{longtable}
\usepackage[colorlinks,linkcolor=blue,anchorcolor=blue,citecolor=green,backref=page]{hyperref}
\theoremstyle{plain}
    \newtheorem{thm}{Theorem}[section]
    
    \newtheorem{claim}[thm]{Claim}
     \newtheorem{conjecture}[thm]{Conjecture}
    \newtheorem{corollary}[thm]{Corollary}
    
    \newtheorem{lemma}[thm]{Lemma}
    \newtheorem{proposition}[thm]{Proposition}
    
    \newtheorem{theorem}[thm]{Theorem}

\theoremstyle{definition}

    \newtheorem*{notation*}{Notation and Terminology}
    \newtheorem{remark}[thm]{Remark}
\theoremstyle{remark}

    \newtheorem{setup}[thm]{}

\newcommand{\C}{\mathbb{C}}
\newcommand{\F}{\mathbb{F}}

\newcommand{\PP}{\mathbb{P}}
\newcommand{\BPP}{\mathbb{P}}
\newcommand{\Q}{\mathbb{Q}}

\newcommand{\R}{\mathbb{R}}

\newcommand{\Z}{\mathbb{Z}}

\newcommand{\SE}{\mathcal{E}}

\newcommand{\SSL}{\mathcal{L}}

\newcommand{\OO}{\mathcal{O}}
\newcommand{\SO}{\mathcal{O}}

\newcommand{\alb}{\operatorname{alb}}

\newcommand{\Aut}{\operatorname{Aut}}

\newcommand{\Bs}{\operatorname{Bs}}
\newcommand{\ch}{\operatorname{char }}

\newcommand{\Gal}{\operatorname{Gal}}

\newcommand{\id}{\operatorname{id}}

\newcommand{\Ker}{\operatorname{Ker}}
\newcommand{\MRC}{\operatorname{MRC}}

\newcommand{\NS}{\operatorname{NS}}

\newcommand{\Proj}{\operatorname{Proj}}

\newcommand{\rank}{\operatorname{rank}}

\newcommand{\Sing}{\operatorname{Sing}}

\newcommand{\torsion}{\operatorname{torsion}}

\newcommand{\variety}{\operatorname{variety}}

\newcommand{\Alb}{\operatorname{Alb}}

\newcommand{\Pic}{\operatorname{Pic}}

\begin{document}
\title[Wild automorphisms of projective varieties]
{Wild automorphisms of projective varieties, the maps which have no invariant proper subsets}

\author{Keiji Oguiso}
\address{Department of Mathematical Sciences, the University of Tokyo, Meguro Komaba 3-8-1, Tokyo, Japan, and National Center for Theoretical Sciences, Mathematics Division, National Taiwan University,
Taipei, Taiwan}
\email{oguiso@g.ecc.u-tokyo.ac.jp}

\author{De-Qi Zhang}
\address{Department of Mathematics, National University of Singapore, 10 Lower Kent Ridge Road, Singapore 119076, Republic of Singapore}
\email{matzdq@nus.edu.sg}


\dedicatory{}

\begin{abstract}
Let $X$ be a projective variety and $\sigma$ a wild automorphism on $X$, i.e.,
whenever $\sigma(Z) = Z$ for a non-empty Zariski-closed subset $Z$ of $X$, we have $Z = X$.
Then $X$ is conjectured to be an abelian variety with $\sigma$ of zero entropy (and proved to be so when $\dim X \le 2$) by
Z.~Reichstein, D.~Rogalski and J.~J.~Zhang in their study of projectively simple rings.
This conjecture has been generally open
for more than a decade.
In this note, we confirm this original conjecture when $\dim X \le 3$ and $X$ is not a Calabi-Yau threefold,
and also show that $\sigma$ is of zero entropy when $\dim X \le 4$ and the Kodaira dimension $\kappa(X) \ge 0$.
\end{abstract}


\thanks{The first named author is partially supported by JSPS and NCTS and the second named author is partially supported by an ARF of NUS}

\maketitle

\tableofcontents

\section{Introduction}
We work over the complex number field
$\C$, unless stated otherwise  (cf.~Proposition \ref{BaseField}).
Let $X$ be a projective variety.
An automorphism $\sigma$ in $\Aut(X)$
is called {\it wild} in the sense of \cite{RRZ} if:
whenever $\sigma(Z) = Z$ for a non-empty
Zariski-closed subset $Z$ of $X$, we have $Z = X$, or equivalently, for every $x \in X$, the orbit
$\{\sigma^s(x) \, | \, s \ge 0\}$ is Zariski dense in $X$.
This always holds when $\dim X = 0$.

In this note, we will prove Theorems \ref{ThA} and \ref{ThB} below which support the following conjecture of
Z.~Reichstein, D.~Rogalski and J.~J.~Zhang (\cite[Conjecture 0.3]{RRZ}).

\begin{conjecture}\label{stronger}
Assume that a projective variety $X$ admits a wild automorphism. Then $X$ is isomorphic to an abelian variety.
\end{conjecture}

Once $X$ is known to be an abelian variety (as asserted in Conjecture \ref{stronger}), the wild automorphisms on $X$
can been completely classified as in \cite[Theorem 0.2]{RRZ};
{\it in particular, they have zero entropy.}
We recall that for a compact K\"ahler manifold (or smooth projective variety) $X$ of dimension $n \ge 1$, and $f \in \Aut(X)$, we let $d_i(f)$ be the $i$-th {\it dynamical degree} of $f$, that is, the spectral radius of $f^*|_{H^{i,i}(X, \R)}$.
The well-known log concavity of dynamical degrees due to Khovanskii and Teissier asserts that
$$(*) \,\,\,\, d_{i-1}(f) d_{i+1}(f) \le d_i(f)^2$$
for all $1 \le i \le n -1$. Hence $d_i(f) = 1$ for one $i$ with $1 \le i \le n-1$ implies that it holds for all such $i$. The classical results of Gromov and Yomdin imply that the {\it topological entropy}
$$h(f) = \log \max_{1 \le i \le n} \{d_i(f)\}.$$
Hence $f$ has zero entropy if and only if $d_1(f) = 1$.

When $\dim X \le 2$, Theorem \ref{ThA} is proved in \cite[Theorem 6.5]{RRZ} (or Theorem \ref{RRZth} for a slightly simplified proof there). See \cite{Ki}, and also our Section \ref{sect_CY} for results when $X$ has Kodaira dimension zero.
Theorem \ref{ThA} follows from Propositions \ref{PropB}, \ref{Prop_q2}, \ref{Prop_q1} and \ref{PropC}.

\begin{theorem}\label{ThA}
Let $X$ be a projective variety over $\C$ of dimension $\le 3$.
Assume that $X$ admits a wild automorphism $\sigma$.
Then either $X$ is an abelian variety, or
$X$ is a Calabi Yau manifold of dimension three
and $\sigma$ has zero entropy.
\end{theorem}

Here, a projective variety $X$ is a {\it Calabi Yau manifold} if its canonical line bundle $\OO(K_X)$
is torsion and its topological fundamental group
is finite.

\begin{remark}\label{CYrem}
We believe that a Calabi-Yau threefold $X$ should not have any wild automorphism $\sigma$. This is confirmed if \cite[Question 2.6]{Og_IJM} (= Conjecture \ref{Og_conj}) is affirmative
for $X$ and this is also equivalent to the projective simplicity of the twisted homogeneous coordinate (non-commutative) ring of any $\sigma$-ample line bundle of $\sigma$ with zero entropy (cf.~\cite[Proposition 2.2]{RRZ} and \cite[Theorem 1.2 (2)]{Ke}). See Section \ref{zero_ent} for details.
\end{remark}

Let $X$ be a projective variety.
As is already indicated in Remark \ref{CYrem}, the motivation of studying wild automorphism arises from \cite[Proposition 0.1]{RRZ} where
it is proved that the {\it twisted homogeneous
coordinate ring}
$$B(X, L, \sigma) = \oplus_{n=0}^{\infty} \, H^0(X, \, \otimes_{i=0}^{n-1} \, (\sigma^i)^*L)$$
associated to any $\sigma$-ample line bundle $L$ on $X$
is projectively simple if and only if $\sigma$ is a wild automorphism on $X$
(cf.~\cite[Proposition 2.2]{RRZ});
every ample line bundle on an abelian variety is $\sigma$-ample if $\sigma$ is wild (\cite[Corollary 8.6]{RRZ}). More generally, by \cite[Theorem 1.2]{Ke}, $X$ admits a $\sigma$-ample line bundle if and only if $\sigma$ has zero entropy, and in this case, any ample line bundle on $X$ is $\sigma$-ample. For this reason, the following weaker conjecture would also be interesting toward Conjecture \ref{stronger} (\cite[Conjecture 0.3]{RRZ}).

\begin{conjecture}\label{weaker}
Every wild automorphism $\sigma$ of a projective variety $X$ has zero entropy.
\end{conjecture}

Theorem \ref{ThB} (1) below (part of Proposition \ref{PropC}),
is also the key in giving a complete proof of Theorem \ref{Thm_Ki}, thus, together with Theorem \ref{ThA}, proving Conjecture \ref{stronger} in dimension three, assuming Conjecture \ref{Og_conj}.

\begin{theorem}\label{ThB}
Conjecture \ref{weaker} is true if either one of the following cases occurs.
\item[(1)]
$\dim X \le 3$.
\item[(2)]
$\dim X = 4$, and the Kodaira dimension $\kappa(X) \ge 0$.
\item[(3)]
$\dim X = 4$, $\kappa(X) = -\infty$, and the irregularity $q(X) \ne 1, 2$.
\end{theorem}

\par \vskip 1pc
\noindent
{\bf Acknowledgement.}
The authors would like to thank
J.~A.~Chen, M.~Satriano, Ch.~Schnell and J.~Xie for the valuable discussions, and the referee for several constructive comments to improve the paper.

\section{Preliminary results in arbitrary dimensions}

In this section, we collect some results mostly noticed in \cite{RRZ}.
We present them in a slightly different and expanded way.
The main new ingredients are Proposition \ref{FL} (Fujiki and Lieberman type result) and Lemma \ref{MRC} on the maximal rationally connected fibration (a special choice of the birational model: Chow reduction) after \cite{Na}. For instance, Proposition \ref{FL} can be used to prove Theorem \ref{ThA} for $\dim X \le 2$ without any detailed analysis of the geometry of ruled surfaces. See the proof of Theorem \ref{RRZth}.

First of all, the following easy remark justifies our assumption on the base field.

\begin{proposition}\label{BaseField}
Let the base field $k$ be an algebraic closure of a finite field $\F_{q}$. Then no automorphism of any projective variety of positive dimension (over $k$) is wild.
\end{proposition}

\begin{proof}
Suppose the contrary that $X$ is a positive dimensional projective variety over $k$ and $\sigma$ is an automorphism of $X$ over $k$. Since $k$ is an algebraic closure of $\F_{q}$, both $X$ and $\sigma$ are defined over some field $k_0$, which is a finite extension of $\F_{q}$. In particular, $k_0$ is a finite field as well. Replacing $k_0$ by a finite extension if necessary, we may assume further that $X(k_0) \not= \emptyset$. Here $X(k_0)$ is the set of $k_0$-rational points of $X$. Since $k_0$ is a finite field, $X(k_0)$ is a finite set. Moreover, since $\sigma$ is defined over $k_0$, we have $\sigma(X(k_0)) = X(k_0)$. Since $X(k_0) \not= X$ from $\dim\, X > 0$, it follows that $\sigma$ is not wild if $\dim X > 0$.
\end{proof}

\begin{proposition}\label{RRZp} (cf.~\cite[Proposition 3.2, Remark 8.3]{RRZ})
Let $X$ be a projective variety.
Suppose that an algebraic
group $G \subseteq \Aut(X)$ acts regularly on $X$ such that $\sigma \in G$ is a wild automorphism.
Then $X$ is an abelian variety and $\sigma$ is a torus translation of infinite order, with the origin of $X$ chosen in any way.
\end{proposition}

The result below is proved in \cite{Fu} and \cite{Li} and generalized in \cite[Theorem 1.2]{DHZ}. See also \cite[Theorem 1.4]{LiSi} for a purely algebraic proof. We note that $\Aut_0(X)$ is a connected group; hence we get the triviality of its action on the (discrete) lattice $\NS(X)/(\torsion)$ as well as its action on $\NS(X)_{\R} = \NS(X) \otimes_{\Z} \R$.

\begin{proposition}\label{FL} (cf.~\cite{Fu}, \cite{Li}, \cite{DHZ}, \cite{LiSi})
Let $X$ be a normal projective variety and
we fix a big Cartier divisor class $[\omega] \in \NS(X)_{\R}$ of $X$.
Then the group
$$\Aut_{[\omega]}(X) := \{g \in \Aut(X) \, | \, g^*[\omega] = [\omega]\}$$
is a finite extension of the identity connected component $\Aut_0(X)$ of $\Aut(X)$, that is, $[\Aut_{[\omega]}(X) : \Aut_0(X)] < \infty$.
\end{proposition}

Since $\Aut_0(X)$ is an algebraic group when $X$ is projective, the above two propositions with Lemma \ref{fin} (2) below imply the
following.

\begin{corollary}\label{FLc}
Let $X$ be a projective variety with a wild automorphism $\sigma$.
Suppose a positive power of $\sigma$ fixes a big divisor class of $X$
(this is the case when the action of $\sigma$ on the N\'eron Severi group $\NS(X)$ is of finite order).
Then $X$ is an abelian variety.
\end{corollary}

We collect some basic properties of wild automorphisms in the two results below.

\begin{lemma}\label{fin} (cf.~\cite[Lemma 3.1, Proposition 4.1, Proposition 5.1]{RRZ})
Let $X$ be a projective variety and let $\sigma$ be an automorphism on $X$.
\begin{itemize}
\item[(1)]
If $\sigma$ is wild then $X$ is smooth.
\item[(2)]
$\sigma$ is wild if and only if so is $\sigma^m$ for some $m \ge 1$ (and hence for all $m \ge 1$).

\item[(3)] Suppose that $\sigma$ is a wild automorphism of $X$ and $f: X \to Y$ (resp. $g: W \to X$ with $g({\rm Sing}\, (W)) \not= X$) is a $\sigma$-equivariant surjective morphism of projective
varieties. Then $f$ (resp. $g$) is a smooth morphism.

\item[(4)]
Suppose that $f : X \to Y$ is a $\sigma$-equivariant surjective morphism to a projective variety $Y$.
If the action of $\sigma$ on $X$ is wild then so is the action of $\sigma$ on $Y$ (and hence $Y$ is smooth).

\item[(5)]
Suppose that $f : X \to Y$ is a $\sigma$-equivariant generically finite surjective morphism of projective varieties.
Then the action of $\sigma$ on $X$ is wild if and only if so is the action of $\sigma$ on $Y$.
Further, if this is the case, then $f: X \to Y$ is a finite \'etale morphism, and in particular,
it is an isomorphism when $f$ is birational.

\item[(6)]
If $X$ is an abelian variety and $\sigma$ is wild then $\sigma$ has zero entropy.
\end{itemize}
\end{lemma}

\begin{proof}
(1) Since $\Sing X$ is stabilized by every automorphism and
$\sigma$ acts as a wild automorphism on $X$, our $X$ is smooth.

For (2), if $\sigma^m$ stabilizes a Zariski-closed subset $Z$ of $X$,
then $\sigma$ stabilizes the closed subset $\cup_{i=0}^{m-1} \, \sigma^i(Z)$ of $X$. Hence (2) is true.

(3) Let $Y_1 \subseteq Y$ be the subset consisting of points $y \in Y$
such that the fibre $X_y = X \times_k k(y)$ is not smooth.
Let $Y_2 \subseteq Y$ be the subset consisting of points $y \in Y$
such that $f : X \to Y$ is not flat at $y$. Since $X$ is smooth,
both $Y_i$ (resp. their inverses in $X$) are Zariski-closed proper subsets of $Y$ (resp. $X$) and they are $\sigma$-stable.
Hence $Y_i = \emptyset$. So $f$ is smooth.
The case of $g$ is similar, by considering the subset of $X$ over which $g$ is non-flat or singular, each of which is a Zariski-closed subset of $W$ being different from $W$ by the generic flatness or by our additional assumption that $g({\rm Sing}\, (W)) \not= X$.

(4) and (5) are similar (and use (3)).

(6) follows from \cite[Theorem 0.2]{RRZ}.
\end{proof}

A smooth projective variety $V$ is a {\it Calabi Yau manifold in the strict sense} if
$V$ is simply conencted, $K_V \sim 0$ and $H^j(V, \OO_{V}) = 0$ ($0 < j < \dim V$).

\begin{proposition}\label{PropA} (cf.~\cite[Lemma 3.1, Corollary 4.2, Proposition 4.4, Remark 4.5]{RRZ})
Let $X$ be a projective variety of dimension $\ge 1$, with a wild automorphism $\sigma$. Then we have:
\begin{itemize}
\item[(1)]
Both the Euler Poincar\'e characteristic and the topological Euler number vanish: $\chi(\OO_X) = 0$, $e(X) = 0$.
In particular, $X$ is not rationally connected.
\item[(2)]
Let $L \in \Pic(X)$ such that $\sigma^*L = L$ in $\Pic(X)$. Then $|L| = \emptyset$ or $L = \OO_X$ in $\Pic(X)$. In particular, the Kodaira dimension $\kappa(X) \le 0$.
\item[(3)]
Suppose that $\kappa(X) = 0$. Then $K_X \sim_{\Q} 0$; the Beauville-Bogomolov (minimal split) finite \'etale cover
$\widetilde{X}$ of $X$ is a product of an abelian variety $A$ of dimension $\ge 0$
and a few copies of
Calabi Yau manifolds $C_i$ of odd dimension $\ge 3$ and in the strict sense; and a positive power of $\sigma$ lifts
to a diagonal action on $\widetilde{X} = A \times \prod_i \, C_i$ whose action on each factor is wild.
\end{itemize}
\end{proposition}

\begin{proof}
(1) See \cite[Proposition 4.4, Remark 4.5]{RRZ}.

If (2) is false, $|L| \simeq \PP^N$ with $N \ge 1$. It follows that the action of $\sigma$ on $|L|$ has a fixed point, say $D \in |L|$. Then $\sigma(D) = D$, contradicting the fact that $\sigma$ is wild.

(3) We have $K_X \sim_{\Q} D$ for some effective $\Q$-divisor (which is unique).
Thus $D$ is $\sigma$-stable. Hence $D = 0$, and $K_X \sim_{\Q} 0$.
Let $\widetilde{X} \to X$ be the Beauville-Bogomolov covering such that
$\widetilde{X}$ is the product of an abelian variety $A$, hyperk\"ahler manifolds $H_i$, and Calabi Yau manifolds $C_j$
in the strict sense. Replacing the cover $\widetilde{X}$ by the minimal splitting cover in \cite[\S 3]{Be}, we can lift $\sigma$ to an action on $\widetilde{X}$ so
that $\widetilde{X} \to X$ is $\sigma$-equivariant. Hence the action of $\sigma$ on
$\widetilde{X}$ is also wild; further,
$\sigma$ (replaced by a positive power)
splits as diagonal actions on the factors $A$,
$H_i$ and $C_j$
(cf.~\cite[\S 3]{Be}),
which are still wild by Lemma \ref{fin}, and hence these factors have
the vanishing Euler Poincare characteristic by (1).
Thus the $H_i$ does not appear, and the $C_j$ are Calabi Yau manifolds of odd dimension ($\ge 3$) in the strict sense.
This proves the proposition.
\end{proof}

The result below uses the standard results on subvarieties of a complex tori in \cite{Ue}.

\begin{lemma}\label{Tori} (cf.~\cite[Corollary 4.3, Proposition 5.1]{RRZ})
Let $X$ be a (smooth) projective variety with a wild automorphism $\sigma$, let
$A$ be an abelian variety
and let $f : X \to A$ be a $\sigma$-equivariant morphism.
Then the image $Y := f(X)$ is a subtorus of $A$ and $f : X \to Y$ is a smooth surjective morphism.
In particular, the albanese map $\alb_X : X \to A = \Alb(X)$
is a surjective smooth morphism with connected fibres.
\end{lemma}

\begin{proof}
It is known that every subvariety of an abelian variety, like $Y = f(X)$ of $A$,
has Kodaira dimension $\kappa(Y) \ge 0$ and equality holds if and only if
$Y$ is a subtorus (after choosing a new origin for $A$); further,
if the inequality is strict then there is a subtorus $B \subset A$ acting on $Y$ such that
$Y \to Y/B$ is birational to the Iitaka fibration, see \cite[Theorem 10.9, Corollary 10.12]{Ue}.
Now our $\sigma$ acts on $Y$ and descends to a birational automorphism $\sigma|_{Y/B}$ on the
base of this Iitaka fibration. By the finiteness of the pluricanonical representation \cite[Theorem 14.10]{Ue}, $\sigma|_{Y/B}$ is of finite order. Hence a sum of finitely many fibres of $X \to Y/B$
is $\sigma$-stable.
Since $\sigma$ is wild, this sum equals $X$ itself. Hence $Y/B$ is a point,
so $Y$ equals $B$, a subtorus of $A$.
The surjective morphism $X \to Y = f(X)$ is smooth by Lemma \ref{fin}.

For the final assertion, note that
$\alb_X : X \to \Alb(X)$ is $\Aut(X)$- and hence $\sigma$-equivariant,
and $\Alb(X)$ is generated by the image $\alb_X(X)$ (an abelian variety),
so $\alb_X(X) = \Alb (X)$.
Let $g : X \to W$ be the Stein factorization of $\alb_X$.
Since $\alb_X$ is $\sigma$-equivariant, so is $g$.
By Lemma \ref{fin}, the induced morphisms $\sigma|_{W}$ and $\sigma|_{\Alb(X)}$ are both wild,
$W$ is smooth,
and the induced morphism $X \to W$ is smooth while $W \to \Alb(X)$ is \'etale.
Hence $W$ is also an abelian variety.
In fact, the surjective \'etale morphism $W \to \Alb (X)$ is an isomorphism by the universality of the albanese morphism. This proves the lemma.
\end{proof}

A {\it maximal rationally connected ($\MRC$) fibration} on a normal projective variety has general fibres $F$ rationally connected in the sense of
Campana and Koll\'ar-Miyaoka-Mori (i.e., every two points on $F$ are connected by
an irreducible rational curve) and it is `maximal' among such (rational) fibrations. See eg. \cite[Definition 5.3, Theorem 5.4]{Ko}.

\begin{lemma}\label{MRC}
Let $X$ be a (smooth) projective variety of positive dimension, with a wild automorphism $\sigma$.
Suppose that $X$ is uniruled. Then we can choose the maximal rationally connected ($\MRC$) fibration
$X \to Y$ to be a well defined $\sigma$-equivariant surjective smooth morphism
with $0 < \dim Y < \dim X$.
Further, the action of $\sigma$ on $Y$ is also wild.
\end{lemma}

\begin{proof}
By Nakayama \cite[Proposition 4.14 or Theorem 4.18]{Na},
we can choose $\MRC$  $X \dashrightarrow Y$ to be a (unique) special Chow reduction
with the graph $\Gamma = \Gamma_{X/Y}$ equi-dimensional over $Y$.
Especially, $\sigma$ on $X$ descends to an automorphism $\sigma_Y$ on $Y$.
The natural birational map $\Gamma \to X$ is $\sigma$-equivariant, and hence it is an isomorphism;
see Lemma \ref{fin}. Thus we may assume that $X = \Gamma \to Y$ is a well-defined surjective morphism.
Hence it is smooth by the same Lemma \ref{fin}. Since $X$ is not rationally connected by
Proposition \ref{PropA}, $Y$ is not a point. Since $X$ is uniruled, $\dim Y < \dim X$.
The action of $\sigma$ on $Y$ is wild by Lemma \ref{fin}. This proves the lemma.
\end{proof}

The following conjecture is known in dimension up to three (cf.~\cite[\S 3.13]{KM} and the references therein).

\begin{conjecture}\label{WAC}
(Weak abundance conjecture) A projective variety is uniruled (i.e., covered by rational curves)
if and only if it has negative Kodaira dimension.
\end{conjecture}

\begin{proposition}\label{red}
Let $X$ be a (smooth) projective variety of dimension $\ge 1$,
with a wild automorphism $\sigma$.
Assume either $\dim X \le 3$ or the weak abundance conjecture in dimension $\le n$.
Then we can choose the maximal rationally connected ($\MRC$) fibration
to be a $\sigma$-equivariant surjective
smooth morphism $f : X \to Y$ (with
every fibre a rationally connected variety of dimension $\ge 0$); $\dim Y > 0$;
the action of $\sigma$ on $Y$ is wild
(and hence $Y$ is smooth); $K_Y \sim_{\Q} 0$; and the Beauville-Bogomolov
finite \'etale cover of $Y$ is a product of an abelian variety of dimension $\ge 0$
and a few copies of Calabi-Yau manifolds of odd dimension $\ge 3$ in the strict sense.
\end{proposition}

\begin{proof}
By Proposition \ref{PropA}, we have the Kodaira dimension $\kappa(X) \le 0$.
If $\kappa(X) = 0$, the proposition follows from Proposition \ref{PropA}.

Thus we may assume that $\kappa(X) < 0$.
Then $X$ is uniruled, by our assumption.
Using Lemma \ref{MRC}, we can choose the maximal rationally connected ($\MRC$) fibration
$X \to Y$ to be a well defined $\sigma$-equivariant surjective smooth (non-trivial) morphism.
By Lemma \ref{fin} the induced action $\sigma |_Y$ on $Y$ of $\sigma$ is also wild.
Thus every fibre of $X \to Y$ is rationally connected, since $\sigma |_Y$ fixes the locus on $Y$
(a non-dense proper subset of $Y$) over which the fibres are not rationally connected.
$Y$ is not uniruled by \cite[Corollary 1.4]{GHS}.
Hence the Kodaira dimension $\kappa(Y) \ge 0$ by
the assumed weak abundance conjecture. Now
the proposition follows from Proposition \ref{PropA}.
\end{proof}

\begin{lemma}\label{red2}
Let $X$ be a (smooth) projective variety with a wild automorphism $\sigma$, let $A$ be
an abelian variety and let
$f : X \to A$ be a $\sigma$-equivariant surjective morphism
with connected fibres of positive dimension. Suppose that a positive power $\sigma^s$ of
$\sigma$ fixes the class of a big divisor $D$ on $A$ in $\NS(A)_{\R} = \NS(A) \otimes_{\Z} \R$ (this holds if $\dim A = 1$ or a positive power of $\sigma$
is a translation on $A$).
Then $-K_F$ is not a big divisor, where $F$ is a general fibre of $f$.
\end{lemma}

\begin{proof}
By Lemma \ref{fin}, $f$ is a smooth morphism.
Suppose the contrary that
$-K_F$ is a big divisor. Then $-K_X$ is relatively big over $A$.
Thus $E := -K_X + mf^*D$ is a big divisor for some $m >> 1$, see (1) in the proof of \cite[Lemma 2.5]{CCP}. Now $\sigma^s$ fixes the big class $[E]$. Thus, by Corollary \ref{FLc},
$X$ is an abelian variety and hence $K_F = K_X|_{F} = 0$, a contradiction.
This proves the lemma.
Indeed, for the claim in the bracket of the statement of the lemma, if $\dim A = 1$,
then $\sigma^{12}$ is a translation on $A$;
if a positive power $\sigma^t$ of $\sigma$ is a translation on $A$, then it fixes
every ample divisor class on $A$.
\end{proof}

The lemma below was proved in \cite{Ki}.

\begin{lemma} (cf. \cite[Lemma 2.8]{Ki}) \label{KiLem}
Let $\sigma$ be a wild automorphism of an abelian variety $A$. Let $H \le \Aut_{\variety}(A)$
be a finite subgroup centralized by $\sigma$: $\sigma h = h \sigma$ for all $h \in H$.
Then $H$ consists of translations of $A$.
\end{lemma}

A projective variety $X$ is called a $\Q$-{\it torus} if it has an abelian variety $A_1$ as
an \'etale finite cover (or equivalently it is the quotient of an abelian variety $A_2$ by a finite group acting freely on $A_2$).

\begin{proposition}\label{Qtorus}
Let $X$ be a projective $\Q$-torus with a wild automorphism $\sigma$.
Then $X$ is an abelian variety.
\end{proposition}

\begin{proof}
Let $A \to X$ be the minimal splitting cover of $X$ in \cite[\S 3]{Be}. Then $\sigma$ lifts to an automorphism on $A$, also denoted as $\sigma$.
Note that
the $\sigma$ on $A$ normalizes $H := \Gal(A/X)$.
Hence $\sigma^{r!}$ centralizes every element of $H$, where $r := |H|$.
Since $\sigma^{r!}$ is still wild by Lemma \ref{fin}, $H$ consists of translations by Lemma \ref{KiLem}. Hence $H = \{id_A\}$ by the minimality of $A \to X$.
Therefore $X = A$ and $X$ is an abelian variety.
\end{proof}

\begin{lemma}\label{kernew}
Let $X$ be a smooth projective variety, $C$ a smooth irreducible projective curve, and
$f: X \to C$ a smooth morphism.
Let $\eta$ be the generic point of $C$ and $X_{\eta} = f^{-1}(\eta)$ the generic fibre.
Suppose $\Pic(X_{\eta}) \simeq \NS (X_{\eta})$ (this is the case when $\Pic^0(X_{\eta}) = 0$, eg. when the geometric generic fibre $X_{\overline{\eta}}$ is rationally connected).
Then $\NS(X)|_{X_{\eta}} = \NS(X_{\eta})$
and the kernel of the surjection
$$
\NS(X)_{\Q}\to \NS(X)_{\Q}|_{X_{\eta}} = \NS(X_{\eta})_{\Q}$$
is spanned by $\Q [F]$, where $[F]$ is the class of closed fibres of $f$.
\end{lemma}

\begin{proof}
Note that algebraic equivalence is preserved under morphisms (hence under restrictions). Thus $\NS(X)|_{X_{\eta}} \subseteq \NS(X_{\eta})$. Let $D$ be a prime divisor on $X_{\eta}$ and let $\overline{D}$ be the Zariski closure of $D$ in $X$. Then $\overline{D}$ is a divisor on $X$ and $D = {\overline{D}}|_{X_{\eta}}$ as divisors. This implies $\NS(X)|_{X_{\eta}} = \NS(X_{\eta})$.

Let $L, M \in \Pic (X)$ such that $L|_{X_{\eta}} \cong M|_{X_{\eta}}$
in $\Pic (X_{\eta})$. It suffices to show that $L = M \otimes f^*N$ for some
$N \in \Pic (C)$.

Since $L|_{X_{\eta}} \cong M|_{X_{\eta}}$, it follows that
$$h^0(X_{\eta}, (L \otimes M^{-1})|_{X_{\eta}}) = 1\,\, ,\,\, h^0(X_{\eta}, {(L \otimes M^{-1})^{-1}}|_{X_{\eta}}) =1\,\, .$$
Then by the upper-semicontinuity theorem,
$$h^0(X_{y}, (L \otimes M^{-1})|_{X_{y}}) \ge 1\,\, ,\,\, h^0(X_{y}, {(L \otimes M^{-1})^{-1}}|_{X_{y}}) \ge 1\,\,$$
for all closed fibres $X_y$. Since $X_y$ is smooth projective, it follows that
$(L \times M^{-1})|_{X_{y}} \simeq \OO_{X_y}$ and therefore $h^0(X_{y},(L \times M^{-1})|_{X_{y}}) = 1$
for all closed fibres $X_y$. Then, since the morphism $f$ is smooth, hence flat, the sheaf $N := f_*(L \otimes M^{-1})$ is an invertible sheaf with base change property $f_*(L \otimes M^{-1}) \otimes \kappa(y) \simeq H^0(X_{y}, (L \times M^{-1})|_{X_{y}})$ by Grauert's theorem.
Since $(L \otimes M^{-1})|_{{X_y}} \simeq \OO_{X_y}$ is globally generated, by Nakayama's lemma, the natural morphism
$$f^*N := f^*f_*  (L \otimes M^{-1}) \to L \otimes M^{-1}$$
is then surjective. Since both sides are invertible sheaves, it follows that $f^*N \simeq L \otimes M^{-1}$, i.e., $L = M \otimes f^*N$ as desired.
\end{proof}

\begin{proposition}\label{irregularity}
Let $X$ be a projective variety  over $\C$ of $\dim X = n$ with a wild automorphism $\sigma$.
Then the albanese morphism $\alb_X : X \to \Alb (X)$ is a smooth surjective morphism with connected fibres and hence $0 \le q(X) = \dim \Alb (X) \le n$.
If $q(X) = n$, then $X = \Alb(X)$, an abelian variety;
if $q(X) \ge n-1$, then $\sigma$ has zero entropy.
\end{proposition}

\begin{proof}
By Lemma \ref{Tori}, $\alb_X : X \to \Alb (X)$ is a smooth surjective morphism with connected fibres.
Assume that $q(X) = n-1$. Then, since the fibre of $\alb_X$ is one dimensional,
$$d_1(\sigma|_X) = d_1(\sigma|_{\Alb (X)}) = 1$$
by the product formula in \cite{DNT} and the fact that $\sigma|_{\Alb (X)}$ is a wild automorphism of an abelian variety and hence has zero entropy (cf.~Lemma \ref{fin}).
The rest of the proposition is clear.
\end{proof}

\begin{proposition}\label{EulerLefschetz}
Let $X$ be a projective variety over $\C$. Suppose $\sigma$ is a wild automorphism of $X$. Put $$A := \oplus_{k \ge 0} \, H^{2k+1}(X, \Z)/(\torsion), \hskip 1pc
B := \oplus_{k \ge 0} \, H^{2k}(X, \Z)/(\torsion).$$
Then the set of complex eigenvalues of $\sigma^*|_{A}$ and the set of complex eigenvalues of $\sigma^*|_{B}$ coincide counted with multiplicities.
\end{proposition}

\begin{proof}
By Lemma \ref{fin}, $X$ is smooth.
By Proposition \ref{PropA}, the topological Euler number $e(X) = 0$. Thus $A$ and $B$ have the same rank, say $r > 0$.  Let $\{a_i \}_{i=1}^{r}$ be the set of complex eigenvalues of $\sigma^*$ on $A$
counted with multiplicities and $\{b_i\}_{i=1}^{r}$ be the set of
 complex eigenvalues of $\sigma^*$ on $B$ counted with multiplicities.

Since $\sigma^n$ is wild, the topological Lefschetz number of $\sigma^n$ is $0$, i.e.,
$$\sum_i a_i^n = \sum_i b_i^n$$
for all $n$, in particular, for all $n=1$, $2$, ... , $r$.
By Newton's identities, this implies that the elementary symmetric polynomials in $a_{1}, \ldots, a_{r}$ coincide with the elementary symmetric polynomials in $b_{1}, \ldots, b_{r}$.
Hence the monic polynomial with roots $\{a_i \}_{i=1}^{r}$ counted with multiplicities and the one with roots $\{b_i \}_{i=1}^{r}$ counted with multiplicities
coincide. Thus $\{a_i \}_{i=1}^{r} = \{b_i\}_{i=1}^{r}$ counted with multiplicities.
\end{proof}

\section{Preliminary results in lower dimensions}

Theorem \ref{RRZth} below was proved in \cite[\S 6]{RRZ}. Our Lemma \ref{red2}
(and Proposition \ref{FL} implicitly)
helps skipping the detailed analysis of ruled surfaces; compare with the proof of \cite[Lemma 6.2]{RRZ}.

\begin{theorem}\label{RRZth} (cf.~\cite[\S 6]{RRZ})
Let $X$ be a projective variety with a wild automorphism $\sigma$.
Suppose that $\dim X \le 2$. Then $X$ is an abelian variety.
\end{theorem}

\begin{proof}
Assume to the contrary that $X$ is not an abelian variety. Then by Proposition \ref{red} and Proposition \ref{Qtorus}, it follows that $\dim X = 2$ and $X$ admits a smooth fibration $f : X \to Y$ with fibres $F$ smooth rational curve and $Y$ an elliptic curve.
But then a general fibre $F$ of $f$ satisfies
$F \cong \PP^1$ has ample $-K_F$, contradicting Lemma \ref{red2}.
\end{proof}

Theorem \ref{ThA} for threefolds will follow from Proposition \ref{PropB} below and Propositions \ref{Prop_q2}, \ref{Prop_q1} and \ref{PropC} in subsequent sections.

\begin{proposition}\label{PropB}
Let $X$ be a projective variety of dimension three with a wild automorphism $\sigma$.
Suppose that $X$ is neither an abelian variety, nor a Calabi-Yau manifold.
Then $X$ is uniruled and the maximal rationally connected ($\MRC$) fibration
can be chosen to be a $\sigma$-equivariant surjective smooth morphism $f : X \to Y$
and coincides with the albanese map $\alb_X: X \to \Alb(X) = \Alb(Y)$,
such that $\sigma$ acts on $Y$ as a wild automorphism and
one of the following cases occurs.
\begin{itemize}
\item[(1)]
The irregularity $q(X) = 1$. In this case, $Y$ is an elliptic curve, $\sigma$ acts on $Y$ as a translation of infinite order and every closed fibre $X_y$ over $y \in Y$ is a smooth rational surface.
\item[(2)]
$q(X) = 2$. In this case, $Y$ is an abelian surface, $\sigma$ acts on $Y$ as a wild automorphism, and every closed fibre $X_y$ over $y \in Y$ is a smooth rational curve: $X_y \cong \PP^1$.
\end{itemize}
\end{proposition}

\begin{proof}
Since $X$ is neither an abelian threefold, nor a Calabi-Yau threefold, by virtue of Proposition \ref{red} and Proposition \ref{Qtorus}, X is uniruled and the $\MRC$ $f: X \to Y$ is a $\sigma$-equivariant surjective
smooth morphism with $Y$ an abelian variety of dimension $1$ or $2$
and every fibre $F$ a smooth rational variety of dimension $2$ or $1$, respectively.
Here we use the fact that
a rationally connected variety of dimension at most two is a rational variety.

Since $Y = \Alb(Y)$, there is a morphism $g : \Alb(X) \to Y$ such that $g \circ \alb_X = f$ by the universal property of the albanese map $\alb_X$. Since $f$ is surjective with connected fibres, so is $g$. Let $X_y$ ($y \in Y$) be any closed fibre of $f$. Then $\alb_X(X_y)$ is a point for each $X_y$, because $X_y$ is rationally connected and any abelian variety contains no rational curves.
Thus $g$ is a finite morphism (with connected fibres), so it is an isomorphism by the normality (and indeed smoothness) of $\Alb(X)$ and $Y$.
This proves the proposition.
\end{proof}

The next result is certainly known at least over $\C$ (cf.~\cite[Table 3.1]{Sa}). For the sake of completeness and its own interest, we shall give a proof which is vailid over any algebraically closed field of characteristic $\not= 2$.

\begin{proposition}\label{EllipticRuled}
Let $\pi : S \to E$ be a relatively minimal elliptic ruled surface over an algebraically closed field of characteristic $\not= 2$. Then the Iitaka $D$-dimension $\kappa(S, -K_S) \ge 0$.
\end{proposition}

\begin{proof}
We write $S = \BPP(\SE) = \Proj(\oplus_{n \ge 0} {\rm Sym}^n(\SE))$ where $\SE$ is a normalized rank $2$ locally free sheaf on the elliptic curve $E$. Then we have an exact sequence
$$0 \to \SO_E \to \SE \to \SSL \to 0$$
for some invertible sheaf $\SSL$ on $E$. Let $C_0$ be the section corresponding to this surjection. Then, with respect to $\SO_{S}(1)$ defined by $\SE$, one has
$\SO_{S}(1) = \SO_{S}(C_0)$ and
$$-K_S = \SO_{S}(2) - \pi^* {\rm det} \, \SE = \SO(2C_0) - \pi^* {\det} \, \SE
= \SO(2C_0) - \pi^* \SSL$$
in $\Pic (S)$ by the canonical bundle formula.

If $\SE$ is decomposable, then ${\rm deg} \SSL \le 0$ and therefore $\SE =\SO_E \oplus \SSL$. Corresponding to the projection $\SE \to \SO_E$, we have a section $C_1$ such that
$\SO_S(C_1) = \SO_{\BPP(\SE)}(1) - \pi^* {\rm det} \, \SE$
in $\Pic (S)$. Hence $-K_S$ is linearly equivalent to an effective divisor $C_0 + C_1$.

If $\SE$ is indecomposable of degree $0$, then $\SE$ is the non-trivial extension of the form
$$0 \to \SO_E \to \SE \to \SO_E \to 0$$
and therefore
$-K_S \cong \SO_S(2)$. Hence $-K_S$ is linearly equivalent to the effective divisor $2C_0$,
where $C_0$ is the section given by the exact sequence above.

It remains to consider the case where $\SE$ is indecomposable of degree $1$.
Our proof below is inspired by \cite{Su}.

Consider the product surface $\BPP^1 \times E$ and its automorphism subgroup $G$ generated by $\tau_1$ and $\tau_2$ defined by
$$\tau_1 : (t, s) \mapsto (-t, s+a), \,\,\,\, \tau_2 : (t, s) \mapsto (t^{-1}, s+b)$$
where $\langle a, b \rangle$ is the group of two torsion points of $E$.
Then $G$ acts freely on $\BPP^1 \times E$ and we have a smooth projective surface $S = (\BPP^1 \times E)/G$. The two projections on $\BPP^1 \times E$ induce on $S$ a relatively minimal elliptic fibre space structure $p_1 : S \to \BPP^1 = \BPP^1/G$ and a relatively minimal elliptic ruled surface structure $p_2 : S \to E \simeq E/G$.

We are going to show that $\kappa(S, -K_S) \ge 0$ and
$p_2 : S \to E \simeq E/G$ is the unique minimal elliptic ruled surface over $E$ corresponding to an indecomposable locally free sheaf of rank $2$ and degree $1$.

Observe that $p_1$ has exactly three multiple fibres of type $_{2}I_{0}$. Hence, by the canonical bundle formula for an elliptic surface applied for $p_1$, it follows that $-2K_S$ is linearly equivalent to $F$, any smooth fibre of $p_1$. In particular, $\kappa(S, -K_S) \ge 0$.

Now it suffices to show that $p_2 : S \to E$ is a ruled surface corresponding to an indecomposable locally free shaef $\SE$ of rank $2$ and degree $1$. Recall that such a minimal elliptic ruled surface is unique from $E$ and does not depend on the choice of such sheaf $\SE$, and this is also the unique case so that $(C_0^2) > 0$,
where $C_0$ is the section of $p_2$ as defined at the beginning of this proof.

Observe that ${\rm deg}\, {p_2} |_{F_t} = 2$ or $4$ for any fibre $F_t$ of $p_1$ with reduced structure. Then the section $C_0$ of $p_2$ is an elliptic curve, but not in the fibre of $p_1$ (whose multiplicity is at most $2$).
Thus $(F.C_0) > 0$ and hence $(C_0^2) > 0$ by
$$0 = {\rm deg} \, K_{C_0} = ((K_S + C_0).C_0)
= -(F.C_0)/2 + (C_0^2)\,\,.$$
This proves the result.
\end{proof}

\begin{lemma} \label{EllipticRuled2}
Let $\pi: S \to E$ be a relatively minimal ruled surface over an elliptic curve $E$ defined over an algebraically closed field of characteristic $\not= 2$.
Then $\kappa(S, -K_S) \ge 0$ and
one of the following cases occurs.
\begin{itemize}
\item[(1)]
$\kappa(S, -K_S) = 2$. In this case for all $t >>1$ and sufficiently divisible, the base locus of $|{-}tK_S|$ has a $1$-dimensional irreducible component.
\item[(2)]
$\kappa(S, -K_S) = 1$. In this case,
$|{-}tK_S|$ is base point free for some $t \ge 1$ and it defines a relatively minimal elliptic fibration $S \to \PP^1$.
\item[(3)]
$\kappa(S, -K_S) = 0$. In this case, there is an integer $m>0$ such that $|{-}mK_S| = \{D\}$ for a non-zero effective divisor $D$ on $S$.
\end{itemize}
\end{lemma}

\begin{proof}
Note that $\kappa(S, -K_S) \ge 0$ by
Proposition \ref{EllipticRuled}  or \cite[Table 3.1]{Sa} (when $k = \C$).

Let $-K_S = P + N$ be the Zariski-decomposition, where $P$ is nef, $P . N = 0$, $N = \sum a_i N_i$ is effective and
has negative intersection matrix $(N_i . N_j)$.
Note that $\kappa(S, -K_S) = \kappa(S, P)$, and the integral part of
$tN$ is contained in the fixed part of $|{-}tK_S|$ for all $t \ge 1$.

If $\kappa(S, -K_S) = 2$, then $-K_S$ and hence $P$ are big. So $P$ is nef and big, and $P^2 > 0$.
Thus $0 = (-K_S)^2 = P^2 + N^2 > N^2$ implies that $N \ne 0$. This is Case(1).

Now we may assume that $\kappa(S, -K_S) = 0$ or $1$. Hence $P$ is not big, so $P^2 = 0$.
Thus $0 = (-K_S)^2 = P^2 + N^2 = N^2$ implies $N = 0$ (the zero divisor). Namely, $-K_S = P$ is nef (but not big).

Suppose that $\kappa(S, -K_S) = 0$.
Then the statement in Case (3) is clear, because $-K_S$, hence $K_S$, is not numerically trivial.

Suppose that $\kappa(S, -K_S) = 1$. Since $-K_S$ is nef and $(-K_S)^2 = 0$,
it follows that $|{-}tK_S|$ is base point free for $t >>1$ and sufficiently divisible.
The adjunction formula for fibre, implies then that $|{-}tK_S|$ defines an elliptic fibration $f: S \to \PP^1$
(with fibres of the ruling $S \to E$ all horizontal to $f$).
This proves the lemma.
\end{proof}

\section{Proof of Theorem \ref{ThA} when $q = 2$}

\begin{proposition}\label{Prop_q2}
The case $q(X) = 2$ in Proposition \ref{PropB} does not occur. (Here our $X$ is defined over $k = \C$, though the argument works over any field $k = \overline{k}$ with $\ch k = 0$.)
\end{proposition}

\begin{proof}
Suppose the contrary that the case $q(X) = 2$ in Proposition \ref{PropB} occurs and
we shall derive a contradiction.

For the $\sigma$-equivariant fibration $f: X \to Y$ there,
let $X_y$ be a fibre over $y \in Y$. Then $X_y \cong \PP^1$.
As in \cite[Theorem 7.2]{RRZ}, write the wild automorphism $\sigma = T_b \circ \alpha$ on $Y$
where $T_b$ is a translation on the abelian surface $Y$
and $\alpha : Y \to Y$ is a group automorphism such that the endomorphism $\beta = \alpha - \id_Y$
is nilpotent.

If $\beta = 0$ then $\sigma = T_b$ and it is a translation; by Lemma \ref{red2},
the anti-canonical divisor $-K_{X_y}$ of a general fibre $X_y$, is not big, which contradicts that $X_y \cong \PP^1$
has ample anti-canonical divisor.

Thus we may assume that $\beta \ne 0$. Let $B$ be a connected component of $\Ker \beta$.
Then $B$ is a nontrivial proper subtorus of the $2$-torus $Y$, and hence an elliptic curve.
Thus $\sigma$ permutes the coset of $E := Y/B$, an elliptic curve.
Hence the quotient map $Y \to E$ is $\sigma$-equivariant.
Since the action of $\sigma$ on $Y$ is wild, so is the action of $\sigma$ on $E$, see Lemma \ref{fin}.
Hence $\sigma$ is a translation of infinite order, see \cite[Remark 8.3]{RRZ}.

Consider the $\sigma$-equivariant fibration $g: X \to E$, which is the composition $X \to Y \to E = Y/B$.
It is a smooth fibration (cf. Lemma \ref{fin}) with each fibre $X_e$ over $e \in E$
a relatively minimal ruled surface over the elliptic curve $Y_e$ (the fibre of $Y \to E$ over $e \in E$),
noting that every fibre of $X \to Y$ is isomorphic to $\PP^1$.

For $-K_{X_e} = (-K_X)|_{X_e}$,
note that $h^0(X_e, \SO(-mK_X)|_{X_e})$ and hence
$\kappa(X_e, -K_{X_e})$ are upper semi-continuous as functions in $e \in E$.
The subsets of points of $E$ over which these functions obtain larger values,
are Zariski closed subsets and $\sigma$-stable. Since $\sigma$ is wild, we may assume that
these functions are all constants as functions in $e \in E$.
Since $g: X \to E$ (a smooth curve) is flat and by Grauert's theorem,
$g_*\SO(-mK_X)$ is locally free and
we have the following natural isomorphism for all $m \ge 0$ and all (closed) point $e \in E$ (with residue field $k(e) = k$):
\begin{equation*} \tag{$*$}
g_* \SO(-mK_X) \otimes k(e) \cong H^0(X_e, \SO(-mK_X)|_{X_e}).
\end{equation*}

For $m>>1$ and sufficiently divisible, consider the rational map
$$h: X \dashrightarrow W = \PP(g_*\SO(-mK_X))$$
associated to the sheaf homomorphism
$$g^*g_* \SO(-mK_X) \to \SO(-mK_X)\, .$$
Since $\sigma^*(-K_X) = -K_X$, there is a natural action of $\sigma$ on $W$
compatible with that on $X$.

If $S = X_e$ has $\kappa(S, -K_S) = 2$ or $0$,
then Lemma \ref{EllipticRuled2} and the isomorphism $(*)$ above imply that the base locus
$\Bs|{-}mK_S|$ is $1$-dimensional for some $m >>1$ and divisible,
whose union
$$\cup_{e \in E} \, \Bs|{-}mK_{X_e}| = {\rm Supp}\, {\rm Coker}\, (g^*g_* \SO(-mK_X) \to \SO(-mK_X))$$
is a $\sigma$-stable Zariski-closed proper subset of $X$, a contradiction.

If $S = X_e$ has $\kappa(S, -K_S) = 1$,
then Lemma \ref{EllipticRuled2} and the isomorphism (*) above imply that
$|{-}mK_S|$
defines an elliptic fibration $S \to \PP^1$  for some $m >>1$ and divisible,
and $h: X \to W$ is a well-defined morphism.
Thus $W$ is a ruled surface, covered by rational curves $h(X_e)$,
and we have a $\sigma$-equivariant
fibration $h: X \to W$. But the action of $\sigma$ on $W$ is also wild by Lemma \ref{fin},
and hence $W$ must be an abelian surface by \cite[\S 6]{RRZ} or Theorem \ref{RRZth}. This is a contradiction.
This proves the proposition.
\end{proof}

\section{Proof of Theorem \ref{ThA} when $q = 1$}

We begin with:

\begin{lemma}\label{Og_lem}
Let $S$ be a smooth projective geometrically rational surface over a field $k$ (not necessarily algebraically closed) and $\Lambda := \NS(S/k)$, the N\'eron-Severi group of $S$ over the field $k$. Let $g \in \Aut (X/k)$ be an automorphism of $S$ over $k$. Assume that $g$ is of zero entropy in the sense that the spectral radius of $g^*|_{\Lambda}$ is one, and that
$g^*|_{\Lambda}$ is of infinite order. Then we have:
\begin{itemize}
\item[(1)]
There is a nonzero pseudo-effective $\Q$-Cartier divisor $v \in \Lambda_{\Q} := \Lambda \otimes_{\Z} \Q$
such that $g^*(v) \sim_{\Q} v$.
Moreover, the ray $\Q_{> 0}v$ of such $v$ is unique.
\item[(2)]
$v$ is nef and $v^2 = 0$.
\item[(3)]
Replacing $v$ by a multiple, we have
\begin{equation*} \tag{$*$}
K_S + v \sim_{\Q} N(S) := \sum_i a_i N_i.
\end{equation*}
Here if $N(S) \ne  0$, then $N(S)$ is an effective divisor
with $N_i$ integral curves,
the intersection matrix $(N_i . N_j)$ is negative
definite and $g^*(N(S)) = N(S)$.
The $v$ in (1) with the extra condition (*) above
is unique up to $\Q$-linear equivalence.
\end{itemize}
\end{lemma}

\begin{proof}
Since $\Lambda$ is hyperbolic, (1) and (2) are proved in the exactly same way as in \cite[Theorem 2.1]{Og}

For (3), by replacing $v$ by a positive multiple, we may assume that $v \in \Lambda$
is (represented by) an integral Cartier divisor.

First note that $K_S . v = 0$. Otherwise, $(K_S + tv)^2 > 0$ for some integer $t$ and $g^*(K_S + tv) = K_S + tv$. Since $(K_S+ tv)^{\perp}$ ($\subset {\Lambda}$) is negative definite, it follows that $g^*|_{\Lambda}$ is of finite order, a contradiction.

By the Serre duality, we also have
$$h^2(K_S + v) = h^0(-v) = 0$$
as $v$ is nef and $v \not= 0$. Hence, by the Riemann-Roch formula, we obtain
$$h^0(K_S + v) \ge \frac{(K_S + v).v}{2} + \chi(\SO_S) = 1.$$
If $K_S + v \sim_{\Q} 0$ in $\Pic(S)_{\Q}$, then (3) clearly holds. If $K_S + v \not\sim 0$ in $\Pic(S)_{\Q}$, then one can take the Zariski decomposition
$$K_S + v = P + N$$
in $\Lambda_{\Q}$ (see eg. \cite[Theorem]{B} whose proof is valid over any field $k$). Now the same proof of \cite[Lemma 4.3, Page 175]{Z-anti_K} implies the result. Indeed, $g^*(P) = P$ in $\Lambda_{\Q}$ and $g^*N = N$ by $g^*(K_S + v) = K_S + v$ and by the uniqueness of the Zariski decomposition. Then, by (1), $P = av$ in $\Lambda_{\Q}$ for some rational number $a$. The rational number $1-a \in \Q$ is positive, since a smooth rational surface $S$ has $K_S$ non-pseudo effective. Replacing $v$ by $(1-a)v$ and putting $N(S) := N$, we obtain (3). Precisely, we have $g^*N = N$ where $N$ is regarded as in $\Lambda_{\Q} \subset \NS(X_{\overline{\eta}})_{\Q} = \Pic(X_{\overline{\eta}})_{\Q}$ (with the last equality due to $S$ being geometrically rational). Since $(N_i . N_j)$ is negative definite
(by the definition of the Zariski decomposition),
we have $g^*N = N$ as Weil $\Q$-divisors.
\end{proof}

\begin{proposition}\label{Prop_q1}
The case $q(X) = 1$ in Proposition \ref{PropB} does not occur.
\end{proposition}

\begin{proof}
Suppose the contrary that the case $q(X) = 1$ in Proposition \ref{PropB} occurs, we shall derive a contradiction.

For the $\sigma$-equivariant fibration $f: X \to Y$ there,
let $X_y$ be a fibre over $y \in Y$. Then $F = X_y$ is a smooth
rational surface and $Y$ is an elliptic curve.
For the wild automorphism action of $\sigma$ on the elliptic curve $Y$,
we have that $\sigma$ is a translation of infinite order.

Let $\eta \in Y$ be the scheme generic point of $Y$ and $X_{\eta} = f^{-1}(\eta)$ the scheme generic fibre of $f$.
Since $\NS(X)|_{X_{\eta}} = \NS(X_{\eta})$ (cf. Lemma \ref{kernew}),
we consider the natural action $\sigma^*$ on
$$\Lambda := \NS(X)|_{X_{\eta}} = \NS(X_{\eta})$$
defined as:
$$\sigma^*(M|_{X_{\eta}}) := (\sigma^*M)|_{X_{\eta}}.$$
We can also define a pairing on $\Lambda$ by
$$({M_1}|_{X_{\eta}} . {M_2}|_{X_{\eta}})_{X_{\eta}}.$$
This paring is compatible with that on $\NS(X_{\eta})$. It is non-degenerate and has signature $(1, r-1)$ with
$r = \rank \Lambda = \rank \NS(X_{\eta})$.

By the flatness of $f$, we have
$$({M_1}|_{X_{\eta}} . {M_2}|_{X_{\eta}})_{X_{\eta}}= ({M_1}|_{F} . {M_2}|_{F})_{F} = (M_1 . M_2 . F)_{X}.$$
Here $F$ is any closed fibre of $f$.

Since $\sigma^*F \equiv F$ (numerical equivalence),
we have
$$(\sigma^*({M_1}|_{X_{\eta}}) . \sigma^*({M_2}|_{X_{\eta}})) = (\sigma^*M_1.\sigma^*M_2.F)_X = (M_1.M_2.F)_X = ({M_1}|_{X_{\eta}} . {M_2}|_{X_{\eta}}).$$
Hence the action of $\sigma$ on $\Lambda$ is an isometry with respect to its pairing and
the spectral radius $\rho({\sigma^*}|_{\, \Lambda})$
is just the first relative dynamical degree $d_1(\sigma|_{f})$ defined in \cite{DNT}.
In particular, we have
$$d_1(\sigma) = \rho({\sigma^*}|_{\, \Lambda})$$
by the product formula, because $Y$ is a curve with $d_1(\sigma|_{Y}) = 1$.

Run the relative minimal model program (MMP) over the elliptic curve $Y$.
Let $X \to X_1$ be the contraction of a $K_X$-negative extremal ray over $Y$.
Then $\rho(X) = \rho(X_1) + 1$.

Suppose that the contraction $X \to X_1$ over $Y$ is a Fano contraction, i.e., a Mori fibre space.
Note that every variety has only finitely many Fano contractions of extremal rays (cf. \cite[Theorem 2.2]{Wi}).
Replacing $\sigma$ by a positive power, we may assume that $\sigma$
stabilizes the extremal ray of the contraction $X \to X_1$
and hence the action of $\sigma$ descends to that on $X_1$
(cf. e.g. \cite[Lemma 2.12]{Z-Compo}) which is again wild, see Lemma \ref{fin}.

If $\dim X_1 \le 1$, then $X_1 = Y$, the fibration $X \to Y$ is a Fano contraction
and $\rho(X) = 2$. Now $\NS(X)_{\Q}$ is spanned by $-K_X$ and a fibre $X_y$ over $y \in Y$
both classes of which are $\sigma$-stable. So ${\sigma^*}|_{\, \NS(X)}$ is of finite order.
Thus $X$ is an abelian variety by Corollary \ref{FLc}, contradicting the existence of a Fano contraction $X \to Y$.

If $\dim X_1 = 2$, then $X_1$ is an abelian surface by Corollary \ref{RRZth}. However, then $1 = q(X) \ge q(X_1) = 2$, a contradiction.

Thus we may assume that $X \to X_1$ is birational.
Since a smooth threefold has no flip, $X \to X_1$ is a divisorial contraction
with the exceptional locus $D_0$, a prime divisor, see \cite{Mo}.
Since we are running the relative MMP of $X$ over the elliptic curve $Y$,
the intersection of $D_0$ with each fibre of $X \to Y$ is a union of
several (rational) curves to be contracted by the map $X \to X_1$.

If the exceptional divisor $D_0$
is $\sigma$-periodic, the union
of $\sigma^i(D_0)$ ($i \ge 0$)
is a $\sigma$-stable Zariski-closed proper subset of $X$, a contradiction.
Therefore, $D_0$ is not $\sigma$-periodic.

Now the proposition follows from the three claims below.

\begin{claim}
$d_1(\sigma |_X) = 1$.
\end{claim}

\begin{proof}
Suppose the contrary that $d_1(\sigma |_X) > 1$.
Then, since the locus $D_0$ of the divisorial extremal contraction $X \to X_1$
is not $\sigma$-periodic,
\cite[Lemma 6.3, Theorem 1.7]{Le} proves the lifting of $\sigma$ to some $X'$
with a $\sigma$-equivariant birational morphism $\rho_{\xi} : X' \to X$,
and the descending of $\sigma$ to some surface $S$
with a $\sigma$-equivariant surjective morphism $\tau : X' \to S$.
Both the action of $\sigma$ on $X'$ and $S$ are wild; see Lemma \ref{fin}.
Now $S$ is an abelian surface by Corollary \ref{RRZth}. However, then
$1 = q(X) = q(X') \ge q(S) = 2$,
a contradiction.
Thus the claim is true. \end{proof}

From now on, we may assume that $\rho({\sigma^*} |_{\Lambda})= d_1(\sigma |_X) = 1$.

\begin{claim}
The action of $\sigma$ on the hyperbolic lattice $\Lambda$ is of infinite order.
\end{claim}

\begin{proof}
Suppose the contrary that $\sigma^* |_{\Lambda}$ is of finite order.
Replacing $\sigma$ by a positive power, we may assume that $\sigma^*$ acts on $\Lambda$ as the identity.
Note that $\sigma^{-1}D_0$ is the exceptional locus of another divisorial contraction of extremal ray on $X$,
see e.g. \cite[Lemma 2.11]{Z-Compo}.
Let
$$X_{\overline{\eta}} = X_{\eta} \times_{{\rm Spec}\, \kappa(\eta)} {\rm Spec}\, \overline{\kappa(\eta)}$$
be the geometric generic fibre of $f$. Since both ${D_0} |_{X_{\overline{\eta}}}$ and $(\sigma^{-1}D_0) |_{X_{\overline{\eta}}}$ are defined over $X_{\eta}$, it follows from $\sigma^* = \id$ on $\Lambda \,$ ($\subset \NS(X_{\overline{\eta}})$) that
$$(D_0) |_{X_{\overline{\eta}}} = (\sigma^*D_0) |_{X_{\overline{\eta}}} \, \in \, \NS(X_{\overline{\eta}})$$

On the other hand, both $(D_0) |_{X_{\overline{\eta}}}$ and $(\sigma^*D_0) |_{X_{\overline{\eta}}}$ are
unions of a few
contractible curves
on the smooth rational projective surface $X_{\overline{\eta}}$.
The negativity of ${D_0}_{|X_{\overline{\eta}}}$ implies that ${D_0} |_{X_{\overline{\eta}}} = (\sigma^{-1}(D_0)) |_{X_{\overline{\eta}}}$ as sets. Then ${D_0} |_{X_{\eta}} = (\sigma^{-1}(D_0)) |_{X_{\eta}}$ as divisors on $X_{\eta}$. By taking the Zariski closure, we obtain $D_0 = \sigma^{-1}(D_0)$, as both $D_0$ and $\sigma^{-1}D_0$ are irreducible divisors on $X$. However, this contradicts the fact that $D_0$ is not $\sigma$-periodic ($\sigma$ being wild). This proves the claim.
\end{proof}

To finish the proof of Proposition \ref{Prop_q1}, we still have to prove the following.

\begin{claim}
It is impossible that the action of $\sigma$ on the hyperbolic lattice $\Lambda$ is parabolic, i.e., it
is of infinite order
and $\rho({\sigma^*} |_{\Lambda})= d_1(\sigma |_X) = 1$.
\end{claim}

\begin{proof}
Suppose the contrary that the action of $\sigma$ on $\Lambda$ is parabolic.
We apply Lemma \ref{Og_lem} to Case $q(X) = 1$ in Proposition \ref{PropB} with
$S = X_{\eta}$ the generic fibre of $f: X \to Y = \Alb(X)$, and $g^*$ equal to
$\sigma^* : \NS(X) \to \NS(X)$
restricted to the subspace $\Lambda  = \NS(X) |_{X_{\eta}} = \NS(X_{\eta})$ of the space $\NS(X_{\overline{\eta}}) =
\Pic(X_{\overline{\eta}})$.
Then, by Lemma \ref{Og_lem}, there is a nonzero nef $\Q$-divisor $v = V |_{X_{\eta}}$ in $\Lambda$ such that $v^2 = 0$,
$(K_X + V)|_{X_{\eta}} \sim_{\Q} N(X_{\eta})$, and
$\sigma^*(N(X_{\eta})) = N( X_{\eta})$. It follows that the Zariski closure of the image of $N(X_{\eta})$ in $X$ is a $\sigma$-stable proper closed subset of $X$. Thus $N(X_{\eta}) = 0$ since the wild automorphism $\sigma$ has no non-empty closed proper subset of $X$. Hence $(K_X + V)|_{X_{\eta}} \sim_{\Q} 0$.
Now $-K_{X_{\eta}} \sim V|_{X_{\eta}}$
is nef with zero self-intersection.
Hence $\kappa(X_{\eta}, -K_{X_{\eta}}) \in \{0, 1\}$
by the Riemann-Roch theorem.

If $\kappa(X_{\eta}, -K_{X_{\eta}}) = 0$,
then $\sigma$ stabilises the unique member in
$|{-}K_{X_{\eta}}|$ and its closure on $X$, contradicting that $\sigma$ is wild.
If $\kappa(X_{\eta}, -K_{X_{\eta}}) = 1$,
then a multiple of the nef $-K_{X_{\eta}}$ supports fibres of the unique (relatively minimal) elliptic fibration on $X_{\eta}$.
This fibration has singular fibres by Kodaira's canonical bundle formula for elliptic (rational) surfaces.
Now $\sigma$ stabilises the union (and its closure in $X$) of singular fibres.
This contradicts that $\sigma$ is wild.
The claim is proved.
\end{proof}

This proves Proposition \ref{Prop_q1}.
\end{proof}

\section{Proof of Theorem \ref{ThB}}\label{sect_CY}

\begin{lemma}\label{cy3}
Let $X$ be a Calabi-Yau manifold of dimension three. Suppose $\sigma$ is a wild automorphism of $X$.
Then $K_X$ is trivial in $\Pic (X)$ and $\sigma$ has zero entropy.
\end{lemma}

\begin{proof}
Since $\chi(\SO_X) = 0$ by Proposition \ref{PropA} and $h^1(\SO_X) = 0$ by assumption, it follows that $h^3(\SO_X) > 0$. Hence $h^0(\SO_X(K_X)) = h^0(\SO_X) > 0$ by the Serre duality. Thus $K_X = 0$ in $\Pic (X)$ as $K_X$ is $\Q$-linearly equivalent to $0$.

We show that $\sigma$ has zero entropy.
Indeed, by \cite{DNT}, replacing $X$ by the universal cover and $\sigma$ by its lifting, we may assume that $X$ is a Calabi-Yau threefold in the strict sense. Since $h^1(\SO_X) = 0$, we have $b_5(X) = b_1(X) = 0$. Thus, under the notation of Propoition \ref{EulerLefschetz}, we have $A = H^3(X, \Z)/(\torsion)$. By Proposition \ref{EulerLefschetz}, it suffices to show that the complex eigenvalues of $\sigma^*|_A$ are roots of the unity.

Again, by $b_5(X) = 0$, the cohomology group $H^3(X, \C)$ is primitive. So is each Hodge component $H^{p,q}(X)$ of $H^{3}(X, \C)$.
Thus the product on $H^{2,1}(X)$ defined below
$$H^{2,1}(X) \times H^{2,1}(X) \ni (\eta, \eta) \mapsto \eta \wedge \overline{\eta}$$
is, up to the following unit (with $(p, q; n, k) = (2, 1; 3, 0)$)
$$c_{p,q; n, k} :=  (\sqrt{-1})^{p-q}(-1)^{(n-k)(n-k-1)/2}, $$
a definite Hermitian form by the Hodge-Riemann relation as in \cite[page 123]{GH}.
Hence the action of $\sigma^*$ on $H^{1,2}(X) \oplus H^{2,1}(X) = H^{1,2}(X) \oplus \overline{H^{1,2}(X)}$ is unitary.
For the same reason, the action of $\sigma^*$ on $H^{3,0}(X) \oplus H^{0,3}(X) = H^{3,0}(X) \oplus \overline{H^{3,0}(X)}$
is also unitary.
So the eigenvalues of the action of $\sigma$ on $H^3(X, \C)$
are of absolute value $1$. Therefore, they are roots of the unity since the action of $\sigma^*$ is defined on the integral $H^3(X, \Z)$ and by Kronecker's theorem. This proves the lemma.
\end{proof}

\begin{lemma}\label{cy4}
Let $X$ be a projective variety of dimension $\le 4$ such that the Kodaira dimension $\kappa (X) \ge 0$.
Suppose $\sigma$ is a wild automorphism of $X$.
Then $\sigma$ has zero entropy.
\end{lemma}

\begin{proof}
By Proposition \ref{PropA}, $K_X \sim_{\Q} 0$ and we may assume that $X$ coincides with its Beauville-Bogomolov minimal split cover and that
each factor $X_i$ of $X$ is stable under $\sigma$ so that $\sigma|_{X_i}$ is wild. Then $\chi(\SO_{X_i}) = 0$ by Proposition \ref{PropA} (1) and also $\dim X_i \le 4$. Therefore $X_i$ is either an abelian variety or a Calabi-Yau threefold
in the strict sense (Proposition \ref{PropA}). Note that $\sigma$ has zero entropy if and only if $\sigma|_{X_i}$ has zero entropy for every factor $X_i$.
If $X_{i}$ is a Calabi-Yau threefold, then $\sigma|_{X_i}$ has zero entropy by Lemma \ref{cy3}.
If $X_{i}$ is an abelian variety, then $\sigma|_{X_i}$ has zero entropy by \cite[Theorem 0.2]{RRZ}.
This proves the lemma.
\end{proof}

\begin{lemma}\label{fourfolds}
Let $X$ be a projective variety over $\C$ of dimension four such that
$q(X) = 0$. Suppose $\sigma$ is a wild automorphism of $X$.
Then $\sigma$ has zero entropy.
\end{lemma}

\begin{proof}
By Lemma \ref{fin}, $X$ is smooth. Let $H^i(X, \Z)_f := H^i(X, \Z)/(\torsion)$ be the free part of $H^i(X, \Z)$.
We need to show the claim that $d_1(\sigma) = 1$.

Note that $\sigma^*$ preserves the closure of the K\"ahler cone of $X$ and the K\"ahler cone spans $H^{1,1}(X, \R)$. Thus, by the Perron-Frobenius theorem,
$\sigma^*L_{\sigma} = d_1(\sigma) L_{\sigma}$ for some nonzero class $L_{\sigma}$ in the closure of the K\"ahler cone of $X$, called a nef class. Of course,
$L_{\sigma} \in H^{1,1}(X, \R) \subset H^2(X, \R)$.

Here and hereafter, we note that $\sigma$ preserves $H^{i,j}(X, \C)$ in the Hodge decomposition
$$H^r(X, \C) = \oplus_{i+j=r} \, H^{i,j}(X, \C),$$
and we denote by
$$H^r(X, K) = H^r(X, \Z) \otimes_{\Z} K$$
for $K = \Q, \R, \C$.

Let $a_i$ ($1 \le i \le r$) be the complex eigenvalues of $\sigma^*|_{H^2(X, \Z)_f}$ counted with multiplicities. Then one of these $a_i$ is $d_1(\sigma)$; this is because
these $a_i$ are also the eigenvalues of $H^2(X, K)$ for
$K = \Q, \R$ or $\C$.
Then the complex eigenvalues of $(\sigma^{-1})^*|_{H^2(X, \Z)_f}$ are $a_i^{-1}$ ($1 \le i \le r$) counted with multiplicities. Since $H^2(X, \Z)_f$ and $H^6(X, \Z)_f$ are dual under the cup product, the complex eigenvalues of $\sigma^*|_{H^6(X, \Z)_f}$ are also $a_i^{-1}$ ($1 \le i \le r$) counted with multiplicities.
Then,
$$1 \le d_1(\sigma) = a_i\,\, ,\,\, 1 \le d_3(\sigma) = d_1(\sigma^{-1}) = a_j^{-1}$$
for some $1 \le i, j \le r$ possibly $i=j$. If $i=j$, then $d_1(\sigma) = 1$. This is because $a_i = 1$
derived from $1 \le a_i$ and $1 \le a_i^{-1}$ if $i=j$.

So, from now on we may assume that $i \not= j$. Then, by renumbering $a_i$ ($1 \le i \le r$) and by interchanging $\sigma$ and $\sigma^{-1}$ if necessary, we may assume that
$$d_1(\sigma) = a_1 \ge a_r^{-1} = d_1(\sigma^{-1}) = d_3(\sigma) \ge 1.$$
Recall that $b_7(X) = b_1(X) = 0$ by the assumption. Then, by Proposition \ref{EulerLefschetz}, $a_1$ is also an eigenvalue of either $\sigma^*|_{H^3(X, \Z)_f}$ or $\sigma^*|_{H^5(X, \Z)_f}$.

Assume first that $a_1$ is an eigenvalue of $\sigma^*|_{H^3(X, \Z)_f}$. Choose an eigenvector $0 \not= u \in H^3(X, \C)$. Since $\sigma^*$ preserves the Hodge decomposition, we may choose $u$ so that $u \in H^{p, q}(X)$ in some Hodge component $H^{p, q}(X)$ of $H^3(X, \C)$. Consider the complex conjugate $\overline{u}$ of $u$ with respect to $H^3(X, \R) \otimes_{\R} \C = H^3(X, \C)$.
Then $\overline{u} \in H^{q, p}(X)$ and $\sigma^*\overline{u} = a_1\overline{u}$, as $\sigma^*$ is defined over $H^3(X, \R)$ and $a_1$ is real.

On the other hand, $H^3(X, \Z)$ is primitive by $b_7(X) = b_1(X) = 0$. Hence the Hodge component $H^{p, q}(X)$ is also primitive.
Then, as in the proof of Lemma \ref{cy3}, we have the unit $c_{p,q, n, k}$ as given there but with
$(p, q; n, k) = (p, q; 4, 1)$ and $p+q = 3$ now,
such that
$$c_{p,q; n,k}(v.\overline{v}.\eta) > 0$$
for every $v \in H^{p,q}(X) \setminus \{0\}$ and for every K\"ahler class $\eta \in H^{1,1}(X, \R)$ by the Hodge-Riemann relation for the primitive cohomology group $H^{p, q}(X)$ (\cite[page 123]{GH}).
In particular,
$$u \overline{u} \not= 0$$
in $H^{3,3}(X, \R)$ and satisfies
$$\sigma^*(u \overline{u}) = a_1^2u \overline{u}.$$
Thus, $a_1^2$ is a positive real eigenvalue of $\sigma^*|_{H^{3,3}(X, \R)}$. Since $d_3(\sigma) = a_r^{-1}$ is the spectral radius of $\sigma|_{H^6(X, \R)}$, it follows that
$$a_1^2 \le d_3(g) = a_r^{-1} \le a_1. $$
For the last inequality, we used our assumption that $a_1 \ge a_r^{-1} \ge 1$. Hence $a_1 = 1$ again, since $1 \le a_1$. Thus $d_1(\sigma) = 1$ as claimed.

Next we consider the case where $a_1$ is an eigenvalue of $\sigma^*|_{H^5(X, \Z)_f}$. Then $a_1$ is an eigenvalue of $(\sigma^{-1})^{*}|_{H^3(X, \Z)_f}$ by the duality.
Then applying the same argument as above for $\sigma^{-1}$, we find that there are a Hodge component $H^{p, q}(X)$ of $H^3(X, \C)$ and a non-zero element $u \in H^{p, q}(X)$ such that $u \overline{u} \in H^{3,3}(X, \R)$ is an eigenvector of $(\sigma^{-1})^{*}|_{H^{3,3}(X, \R)}$ corresponding to an eigenvalue $a_1^2$, which is real and positive. Thus,
$$a_1^2 \le d_3(\sigma^{-1}) = d_1(\sigma) = a_1.$$
Since $1 \le a_1$, it follows that $a_1 = 1$. Hence $d_1(\sigma) = 1$ also in this case, as claimed.

This proves the lemma.
\end{proof}

\begin{proposition}\label{PropC}
Let $X$ be a projective variety over $\C$ with a wild automorphism $\sigma$.
\begin{itemize}
\item[(1)] Suppose $\dim X \le 3$. Then $X$ is either an abelian variety or a Calabi-Yau threefold
with the canonical divisor $K_X \sim 0$.
Moreover $\sigma$ has zero entropy.

\item[(2)] Suppose $\dim X = 4$. Then $\sigma$ has zero entropy, unless the Kodaira dimension
$\kappa(X) = - \infty$ and the irregularity $q(X) >0$.
\end{itemize}
\end{proposition}

\begin{proof}
By Theorem \ref{RRZth} and Propositions \ref{PropB}, \ref{Prop_q2} and \ref{Prop_q1}, we already know that a projective
variety $X$ with $\dim X \le 3$ and with a wild automorphism is isomorphic to either an abelian variety or a Calabi-Yau threefold.
Thus Proposition \ref{PropC} (1) follows from Lemma \ref{cy3}.
Proposition \ref{PropC} (2) follows from Lemmas \ref{cy4} and \ref{fourfolds}.
\end{proof}

\begin{setup}
{\bf Proof of Theorem \ref{ThB}.}
It follows from Proposition \ref{PropC}
and Proposition \ref{irregularity}.
\end{setup}

\section{Wild automorphisms of zero entropy}\label{zero_ent}

In this section, we consider the case where $X$ is a (smooth) projective variety with a wild automorphism
$\sigma$ such that $\sigma$ is of zero entropy.

\begin{lemma}\label{homogrings}
Let $\sigma$ be an automorphism of zero entropy of a smooth projective variety. Then the following three conditions are equivalent.
\begin{itemize}
\item[(1)] $\sigma$ is wild.
\item[(2)] The twisted homogeneous coordinate ring $B(X, L, \sigma)$ is projectively simple for any $\sigma$-ample line bundle $L$ on $X$.
\item[(3)] The twisted homogeneous coordinate ring $B(X, L, \sigma)$ is projectively simple for at least one ample line bundle $L$ on $X$.
\end{itemize}
\end{lemma}

\begin{proof} By \cite[Theorem 1.2 (2)]{Ke}, any ample line bundle on $X$ is $\sigma$-ample, because $\sigma$ is of zero entropy. Thus, the implication (2) $\Rightarrow$ (3) is clear and the implications (1) $\Rightarrow$ (2) and (3) $\Rightarrow$ (1) follow from \cite[Proposition 0.1]{RRZ}.
\end{proof}

\begin{lemma}\label{non-diag}
Let $X$ be a (smooth) projective variety. Suppose that $\sigma$ is a wild automorphism of $X$ of zero entropy. Then $\sigma^*|_{\NS(X)_{\C}}$ is not diagonalizable, unless $X$ is isomorphic to an abelian variety.
\end{lemma}

\begin{proof}
We embed the lattice $\Lambda := \NS(X)/(\torsion)$ in $H^{1,1}(X\ \R) \subset H^{1,1}(X, \C)$.
By the assumption, $d_1(\sigma) = 1$, so the spectral radius of the pullback action $\sigma^*$ on $H^{1,1}(X, \C)$ and hence that on $\Lambda$ are $1$.
Thus every eigenvalue of $\sigma^*|_{\Lambda}$ is an algebraic integer of modulus $1$. Hence, by Kronecker's theorem, all eigenvalues of $\Lambda$, or equivalently of $\sigma^*|_{\NS(X)_{\C}}$, are roots of unity.

So, if $\sigma^*|_{\NS(X)_{\C}}$ is diagonalizable, then $(\sigma^*)^m = \id$ on $\NS(X)$ for some $m > 0$. In this case, $X$ is an abelian variety by Corollary \ref{FLc}. This proves the lemma.
\end{proof}

The following conjecture is in \cite[Question 2.6]{Og_IJM}.

\begin{conjecture}\label{Og_conj} (cf. \cite{Og_IJM})
Every nef $\Q$-Cartier divisor $L$ (not necessarily effective) on a Calabi-Yau manifold of dimension three is $\Q$-linearly equivalent to an effective divisor (and hence semi-ample by the known log Abundance in dimension $\le 3$).
\end{conjecture}

The following is asserted in \cite{Ki}.

\begin{theorem}\label{Thm_Ki} (cf. \cite[Theorem 4.7]{Ki}, and also Remark \ref{rem:Ki})
Assume Conjecture \ref{Og_conj} for $3$-dimensional Calabi-Yau manifolds in the strict sense. Let $X$ be a (smooth) Calabi-Yau threefold. Then $X$ does not admit a wild automorphism.
\end{theorem}

\begin{remark}\label{rem:Ki}
About the proof of \cite[Theorem 4.7]{Ki}, and the necessity of $\Q$-Cartier assumption on $L$ in Conjecture \ref{Og_conj}, we have:

\begin{enumerate}
\item[(1)]
In the proof of \cite[Theorem 4.7]{Ki}, the author proceeds by asserting the existence of a nef divisor $L$ with $c_2(X) . L = 0$ and claims that such $L$ is semi-ample by applying Conjecture \ref{Og_conj}.
This argument does not seem to work because such $L$ may not be a $\Q$-divisor (as required) to apply the conjecture.
\item[(2)]
Indeed, Conjecture \ref{Og_conj} does not hold for
$\R$-Cartier divisors. Precisely, in \cite[Theorem 1.4, Proposition 4.4]{OT}, the authors constructed an automorphism $f$ of positive entropy on a (smooth) Calabi-Yau $3$-fold $X$ such that $f$ is primitive, i.e., there is no $f$-equivariant rational map to any variety $Y$ with
$0 < \dim Y < \dim X$. By Perron-Frobenius theorem, there is a nonzero nef $\R$-Cartier divisor $L$ such that $f^*L \equiv d_1(f) L$. Since $d_1(f) > 1$, we have
$c_2(X) . L = 0$. But this $L$ is not semi-ample, or else it would have induced an $f$-equivariant non-trivial morphism to a curve or a surface, contradicting the primitivity of $f$.

\item[(3)]
For the sake of completeness, in the next paragraphs, we supply a proof of \cite[Theorem 4.7]{Ki}.
\end{enumerate}
\end{remark}

\begin{lemma}\label{nef_Q}
Let $X$ be a normal projective variety and
$\sigma$ an automorphism.
Suppose $\sigma^*|_{\NS(X)_{\C}}$ is of infinite order and
$\sigma$ is of zero entropy.
Then there is a nef (integral) Cartier divisor $L \not\equiv 0$ such that
$\sigma^*L \equiv L$ (numerical equivalence) possibly after replacing $\sigma$ by a positive power.
In particular, we may assume $\sigma^*L \sim L$
if the irregularity $q(X) = 0$.
\end{lemma}

\begin{proof}
Replacing $\sigma$ by a positive power, we may assume that $\sigma^*|_{\NS(X)_{\C}}$ is unipotent
(but not trivial)
so it is equal to $I + N$ with $N$ nilpotent.
Let $m$ ($\ge 2$) be the maximal order of Jordan blocks of $\sigma^*|_{\NS(X)_{\C}}$.
Then $N^m = 0$, but $N^{m-1} \ne 0$. Thus
$$g := \lim_{s \to \infty}
(\sigma^*)^s/s^{m-1} = N^{m-1}/(m-1)!$$
is a nonzero linear transformation on $\NS(X)_{\Q}$.
Since ample Cartier divisor classes generate $\NS(X)_{\Q}$ and since $g = N^{m-1}/(m-1)! \not= 0$ on $\NS(X)_{\Q}$, there is an ample Cartier divisor $H$ on $X$ so that $L := g(H)$ is nonzero $\Q$-Cartier divisor. Since $L_s := (\sigma^*)^s(H)/s^{m-1}$ is ample, $L = \lim_{s \to \infty} L_s$ is nef. Moreover, using $\sigma^* = I +N$ and $N^m = 0$, we compute
$$\sigma^*L = \lim_{s \to \infty}
(I+N)^{s+1}(H)/s^{m-1} = (N^{m-1}/(m-1)!)(H) = g(H) = L .$$
The lemma follows by replacing $L$ by a positive multiple.
\end{proof}

\begin{setup}
{\bf Proof of Theorem \ref{Thm_Ki}.}
Let $X$ be a (smooth) Calabi-Yau threefold with a wild automorphism $\sigma$.
By Proposition \ref{PropC}, $\sigma$ is of zero entropy.
We are going to reach a contradiction.
By Proposition \ref{PropA}, we can lift $\sigma$ to $\widetilde{\sigma}$ on the universal cover $\widetilde{X}$ of $X$,
such that $\widetilde{X}$ is a Calabi-Yau manifold in the strict sense (while $\widetilde{\sigma}$ is, of course, still of zero entropy). Replacing $X$ by $\widetilde{X}$ we may assume that $X$ is a Calabu-Yau manifold in the strict sense.

By Lemma \ref{nef_Q}, we have an
integral nef Cartier divisor $L \not\equiv 0$ such that $\sigma^*L \sim L$.
Since we have assumed that Conjecture \ref{Og_conj} holds for such $X$,
we have $|mL| \not= \emptyset$ for some positive integer $m$, contradicting Proposition \ref{PropA} (2).
This proves Theorem \ref{Thm_Ki}.
\end{setup}

\end{document}